\documentclass[11pt]{article}
\usepackage{amsmath}
\usepackage{amsthm}
\usepackage{amsfonts} 

\usepackage{mathabx}
\usepackage{amssymb}
\usepackage[pdftex]{graphicx} 
\usepackage[english]{babel} 
\usepackage[colorlinks = true,
            linkcolor = blue,
            urlcolor  = blue,
            citecolor = blue,
            anchorcolor = blue]{hyperref} 
\usepackage{calc} 
\usepackage{dsfont}
\usepackage{bbm}
\usepackage{enumitem}

\usepackage{authblk}
\usepackage{physics}
\allowdisplaybreaks
\usepackage[a4paper, lmargin=0.1666\paperwidth, rmargin=0.1666\paperwidth, tmargin=0.1111\paperheight, bmargin=0.1111\paperheight]{geometry} 
\usepackage[all]{nowidow}
\usepackage{lipsum} 
\hypersetup{ 	
	pdfsubject = {},
	pdftitle = {},
	pdfauthor = {}
}

\theoremstyle{plain}
\newtheorem{theorem}{Theorem}[section]
\newtheorem{conjecture}[theorem]{Conjecture}

\theoremstyle{definition}

\theoremstyle{remark}

\newtheoremstyle{named}{}{}{\itshape}{}{\bfseries}{.}{.5em}{\thmnote{#3}#1}
\theoremstyle{named}

\title{On a conjecture of Knuth about forward and back arcs}
\author{Zipei Nie \thanks{Lagrange Mathematics and Computing Research Center, Huawei. Email: niezipei@huawei.com.}}
	
\date{\today}

\begin{document} 
\maketitle
\begin{abstract}
    Following Janson's method, we prove a conjecture of Knuth: the numbers of forward and back arcs for the depth-first search (DFS) in a digraph with a geometric outdegree distribution have the same distribution.
\end{abstract}
\section{Introduction}
A depth-first search (DFS) in a graph starts at an unvisited vertex $v$ and explores each arc from $v$. Each time an arc points to an unvisited vertex $u$, we recursively perform DFS from $u$. Repeat this process until all vertices have been visited.

Each arc is explored exactly once during the DFS. If an arc points to an unvisited vertex in the algorithm, it is called a tree arc. The set of tree arcs forms a spanning forest of the graph, which is called the depth-first forest. We can further classify the remaining arcs. Let $v\to u$ be an arc which is not a tree arc. It is called
\begin{enumerate}[label=(\alph*)]
    \item a loop, if $u=v$;
    \item a forward arc, if $u$ is a descendant of $v$ in the depth-first forest;
    \item a back arc, if $u$ is an ancestor of $v$ in the depth-first forest;
    \item a cross arc, otherwise.
\end{enumerate}

For an integer $n$ and a real parameter $p\in (0,1)$, a digraph with a geometric outdegree distribution is a random multidigraph on $n$ vertices generated in the following way. Independently for each vertex $v$, let $d^+(v)$ be a geometrically distributed random variable with mode $0$ and mean $\frac{p}{1-p}$. Then we create $d^+(v)$ arcs $v\to u$ from each vertex $v$, where these $u$'s are independent uniform random vertices.

In a new section of the book \emph{The Art of Computer Programming}, Knuth proposed the following conjecture on the numbers of forward and back arcs for the DFS in digraphs with geometric outdegree distributions.

\begin{conjecture}\cite[Problem 7.4.1.2-35]{Knuth12A}\label{Conj-1}
Let $F$ and $B$ denote the numbers of forward and back arcs respectively for the DFS in a digraph with a geometric outdegree distribution. Then $F$ and $B$ have the same distribution.
\end{conjecture}

Knuth verified Conjecture \ref{Conj-1} for $n\le 9$ and claimed that ``it \emph{must} be true''. Jacquet and Janson \cite[Theorem 13]{SJ364} proved that $\mathbb{E}(F)=\mathbb{E}(B)$. Janson proposed the following extension of Conjecture \ref{Conj-1}.

\begin{conjecture}\cite[Conjecture 2.1]{J}\label{Conj-2}
Let $L, F,B,C,T$ denote the numbers of loops, forward, backward, cross, and tree arcs respectively for the DFS in a digraph with a geometric outdegree distribution. Then $(L,F,B+C,T)$ and $(L,B,F+C,T)$ have the same distribution.
\end{conjecture}

Janson verified Conjecture \ref{Conj-2} for $n\le 18$. He also showed that Conjecture \ref{Conj-2} is equivalent to an identity involving recursive functions. 

\begin{theorem}\cite[Proposition 3.2 and Remark 3.6]{J}\label{J-thm}
Define $\widehat{G}_n(w,x,z)$ and $\widecheck{G}_n(w,x,z)$ by the recursions
$$\widehat{G}_1(w,x,z)=\widecheck{G}_1(w,x,z)=\frac{1}{1-w},$$
and, for $n\ge 2$,
$$\widehat{G}_n(w,x,z)=\sum_{k=1}^{n-1}\frac{\widehat{G}_{k}(w,x,z)\widehat{G}_{n-k}(w+kz,x,z)}{1-w-(n-1)x},$$
$$\widecheck{G}_n(w,x,z)=\sum_{k=1}^{n-1}\frac{\widecheck{G}_{k}(w+x,x,z)\widecheck{G}_{n-k}(w+kz,x,z)}{1-w}.$$
Then Conjecture \ref{Conj-2} holds if and only if $\widehat{G}_n(w,x,z)=\widecheck{G}_n(w,x,z)$ for all $n\ge 1$.
\end{theorem}

In this note, we prove the identity $\widehat{G}_n(w,x,z)=\widecheck{G}_n(w,x,z)$. As a corollary, Knuth's conjecture and its extension are indeed true.

\begin{theorem}\label{main-thm}
    Define $\widehat{G}_n(w,x,z)$ and $\widecheck{G}_n(w,x,z)$ as in Theorem \ref{J-thm}. Then $$\widehat{G}_n(w,x,z)=\widecheck{G}_n(w,x,z)$$ for all $n\ge 1$.
\end{theorem}

The key to our proof is a stronger identity with one more parameter, namely Theorem \ref{thm2.3}. As described \cite{Polya} in Pólya's book \emph{How to Solve It}, ``the more ambitious plan may have more chances of success.'' The existence of such a stronger identity leads to an elegant proof by induction.

\section{The proof of the identity}
Throughout the section, we treat $x$ and $z$ as constant formal variables. 

Define the function $G_n(w)$ by 
$$G_n(w)= \widehat{G}_n(1-w,-x,-z).$$
Then we have recursive formulas
$$G_1(w)=\frac{1}{w},$$
and, for $n\ge 2$,
$$G_n(w)=\sum_{k=1}^{n-1}\frac{G_k(w)G_{n-k}(w+kz)}{w+(n-1)x}.$$

Let $k$ and $n$ be two positive integers with $k< n$. Define the function $F_{k,n}(w)$ by
$$F_{1,n}(w)=G_{n}(w),$$
and, for $k\ge 2$,
\begin{align*}
    F_{k,n}(w)=& \sum_{i=1}^{n-k-1}\frac{G_{n-k-i}(w+(k+i)z)F_{k,k+i}(w)}{w+(n-1)x}\\&+\sum_{i=1}^{k-1} \frac{(w+iz)G_i(w+x) F_{k-i,n-i}(w+iz)}{w(w+(n-1)x)}.
\end{align*}

First, we derive a formula for $F_{2,n}$.
\begin{theorem}\label{thm2.1}
For each $n\ge 3$, we have $$F_{2,n}(w)= G_n(w) -\frac{G_{n-1}(w+z)}{w(w+x)}.$$
\end{theorem}
\begin{proof}
We induct on $n$. Suppose $n\ge 3$ and that the statement holds for all values less than $n$. Then by definitions of $F_{k,n}$ and $G_n$, we have
\begin{align*}
    &F_{2,n}(w)\\
    = &\sum_{i=1}^{n-3}\frac{G_{n-2-i}(w+(i+2)z)F_{2,i+2}(w)}{w+(n-1)x}+\frac{(w+z)G_1(w+x) F_{1,n-1}(w+z)}{w(w+(n-1)x)}\\
    =&\sum_{i=1}^{n-3}\frac{G_{n-2-i}(w+(i+2)z)}{w+(n-1)x} \left(G_{i+2}(w) -\frac{G_{i+1}(w+z)}{w(w+x)}\right)\\
    &+\frac{(w+z)G_{n-1}(w+z)}{w(w+x)(w+(n-1)x)}\\
    =&\sum_{i=3}^{n-1}\frac{G_{n-i}(w+iz)G_{i}(w)}{w+(n-1)x}-\sum_{i=2}^{n-2}\frac{G_{n-1-i}(w+(i+1)z)G_{i}(w+z)}{w(w+x)(w+(n-1)x)}\\
    &+\frac{(w+z)G_{n-1}(w+z)}{w(w+x)(w+(n-1)x)}\\
    =&\frac{(w+(n-1)x) G_n(w) -G_{n-2}(w+2z)G_2(w)-G_{n-1}(w+z)G_1(w)}{w+(n-1)x}\\
    &-\frac{ (w+z+(n-2)x)G_{n-1}(w+z)- G_{n-2}(w+2z)G_1(w+z)}{w(w+x)(w+(n-1)x)}\\
    &+\frac{(w+z)G_{n-1}(w+z)}{w(w+x)(w+(n-1)x)}\\
    =&G_n(x) -\frac{G_{n-1}(w+z)}{w(w+x)}.
\end{align*}
Therefore, the statement holds for every $n\ge 3$ by the principle of induction.
\end{proof}

Second, we derive a formula for $F_{n-1,n}$.

\begin{theorem}\label{thm2.2}
    For each $n\ge 2$, we have $$F_{n-1,n}(w)=\frac{ G_{n-1}(w+x)}{w(w+(n-1)z)}.$$
\end{theorem}
\begin{proof}
    We induct on $n$. The base case $n=2$ can be checked directly. Suppose that $n\ge 3$ and that the statement holds for all values less than $n$. Then by definitions of $F_{k,n}$ and $G_n$, we have
    \begin{align*}
    &F_{n-1,n}(w)\\
    =& \sum_{i=1}^{n-2} \frac{(w+iz)G_i(w+x) F_{n-1-i,n-i}(w+iz)}{w(w+(n-1)x)}\\
    =& \sum_{i=1}^{n-2} \frac{G_i(w+x) G_{n-i-1}(w+x+iz) }{w(w+(n-1)x)(w+(n-1)z)}\\
    =&\frac{ G_{n-1}(w+x)}{w(w+(n-1)z)}.
    \end{align*}
Therefore, the statement holds for every $n\ge 2$ by the principle of induction.
\end{proof}

Finally, we derive a formula for the difference between $F_{k,n}$ and $F_{k+1,n}$.
\begin{theorem}\label{thm2.3}
    For each $1\le k\le n-2$, we have $$F_{k,n}(w)-F_{k+1,n}(w)= \frac{G_{k}(w+x)G_{n-k}(w+kz)}{w}.$$
\end{theorem}
\begin{proof}
    We induct on $k+n$. By Theorem \ref{thm2.1}, the statement holds for $k=1$. Suppose that $2\le k\le n-2$ and that the statement holds for all values less than $k+n$. Then, by Theorem \ref{thm2.2}, and the definitions of $F_{k,n}$ and $G_n$, we have
    \begin{align*}
        &F_{k,n}(w)-F_{k+1,n}(w)\\
        =&\sum_{i=0}^{n-k-2}\frac{G_{n-k-i-1}(w+(k+i+1)z)F_{k,k+i+1}(w)}{w+(n-1)x}\\
        &-\sum_{i=1}^{n-k-2}\frac{G_{n-k-i-1}(w+(k+i+1)z)F_{k+1,k+i+1}(w)}{w+(n-1)x}\\
        &+\sum_{i=1}^{k-1} \frac{(w+iz)G_i(w+x) F_{k-i,n-i}(w+iz)}{w(w+(n-1)x)}\\&-\sum_{i=1}^{k} \frac{(w+iz)G_i(w+x) F_{k-i+1,n-i}(w+iz)}{w(w+(n-1)x)}\\
        =&\sum_{i=1}^{n-k-2}\frac{G_{n-k-i-1}(w+(k+i+1)z)}{w+(n-1)x}\left(F_{k,k+i+1}(w)-F_{k+1,k+i+1}(w)\right)\\
        &+\sum_{i=1}^{k-1} \frac{(w+iz)G_i(w+x)}{w(w+(n-1)x)}\left( F_{k-i,n-i}(w+iz)- F_{k-i+1,n-i}(w+iz)\right)\\
        &+\frac{G_{n-k-1}(w+(k+1)z)F_{k,k+1}(w)}{w+(n-1)x}-\frac{(w+kz)G_k(w+x) F_{1,n-k}(w+kz)}{w(w+(n-1)x)}\\
        =&G_{k}(w+x)\sum_{i=2}^{n-k-1}\frac{G_{n-k-i}(w+(k+i)z)G_{i}(w+kz)}{w(w+(n-1)x)}\\
        &+G_{n-k}(w+kz)\sum_{i=1}^{k-1} \frac{G_i(w+x)G_{k-i}(w+x+iz)}{w(w+(n-1)x)}\\
        &+\frac{G_{n-k-1}(w+(k+1)z)G_{k}(w+x)}{w(w+(n-1)x)(w+kz)}-\frac{(w+kz)G_k(w+x) G_{n-k}(w+kz)}{w(w+(n-1)x)}\\
        =&G_{k}(w+x)\frac{(w+kz+(n-k-1)x)G_{n-k}(w+kz) -\frac{G_{n-k-1}(w+(k+1)z)}{w+kz}}{w(w+(n-1)x)}\\
        &+G_{n-k}(w+kz)\frac{(w+kx)G_{k}(w+x)}{w(w+(n-1)x)}\\
        &+\frac{G_{n-k-1}(w+(k+1)z)G_{k}(w+x)}{w(w+(n-1)x)(w+kz)}-\frac{(w+kz)G_k(w+x) G_{n-k}(w+kz)}{w(w+(n-1)x)}\\
        =&\frac{G_k(w+x) G_{n-k}(w+kz)}{w}.
    \end{align*}
Therefore, the statement holds by the principle of induction.
\end{proof}

Now we prove our main theorem.
\begin{proof}[Proof of Theorem \ref{main-thm}]
    We induct on $n$. The base cases $n\le 2$ can be checked directly. Suppose that $n\ge 3$ and that the statement holds for all values less than $n$. Then, by Theorem \ref{thm2.1}, Theorem \ref{thm2.2} and Theorem \ref{thm2.3}, we have
    \begin{align*}
        &\widecheck{G}_{n}(1-w,-x,-z)\\
        =&\sum_{k=1}^{n-1}\frac{\widecheck{G}_{k}(1-w-x,-x,-z)\widecheck{G}_{n-k}(1-w-kz,-x,-z)}{w}\\
        =&\sum_{k=1}^{n-1}\frac{\widehat{G}_{k}(1-w-x,-x,-z)\widehat{G}_{n-k}(1-w-kz,-x,-z)}{w}\\
        =&\sum_{k=1}^{n-1}\frac{G_k(w+x)G_{n-k}(w+kz)}{w}\\
        =&G_n(w)-F_{2,n}(w)+\sum_{k=2}^{n-2}\left(F_{k,n}(w)-F_{k+1,n}(w)\right)+F_{n-1,n}\\
        =&G_n(w)\\
        =&\widehat{G}_{n}(1-w,-x,-z).
    \end{align*}
Therefore, the statement holds for every $n$ by the principle of induction.
\end{proof}

\end{document}